\documentclass{article}
\usepackage[utf8]{inputenc}

\usepackage[style=alphabetic,natbib,giveninits=true,doi=false,isbn=false,url=false,eprint=false]{biblatex}
\AtEveryBibitem{\clearlist{language}}
\AtEveryBibitem{\clearfield{pages}} 
\usepackage{soul}
\usepackage{amsmath}
\usepackage{amssymb}
\usepackage{amsthm}
\usepackage{mathtools}

\usepackage{pgf,tikz,pgfplots}

\usepackage{tikz-cd}
\usetikzlibrary{patterns}
\usetikzlibrary{spy}
\usetikzlibrary{arrows,decorations.markings}
\usetikzlibrary{math}

\usepackage{dirtytalk}

\usepackage{mathrsfs}
\usepackage{enumerate}
\usepackage[shortlabels]{enumitem}

\usepackage{subfigure}

\usepackage[margin=1in]{geometry}

\usepackage{xcolor}

\usepackage{setspace}
\theoremstyle{plain}
\newtheorem{thm}{Theorem}

\theoremstyle{plain}
\newtheorem{lemma}{Lemma}

\theoremstyle{plain}

\theoremstyle{plain}

\theoremstyle{plain}
\newtheorem{fact}{Fact}

\theoremstyle{definition}

\theoremstyle{definition}

\usepackage[T2A,T1]{fontenc}

\usepackage[symbol]{footmisc}

\usepackage{array}
\usepackage{makecell}

\usepackage{todonotes}

\usepackage{subfiles}

\usepackage{multicol}
\providecommand{\Acknowledgements}[1]
{
  \small	
  \textbf{\textit{Acknowledgements---}} #1
}

\title{The Bernoulli property for counter-twisting linked twist maps}

\author{J. Myers Hill  \\
         \small School of Mathematics, University of Leeds, Leeds LS2 9JT, United Kingdom \\
         \small E: j.d.myershill@leeds.ac.uk
}

\pgfplotsset{compat=1.17}

\begin{document}

\date{}

\maketitle

\begin{abstract}
We prove the Bernoulli property for a class of counter-twisting linked twist maps. These compose orthogonal linear shears on the torus, orientated in the opposite sense to their co-twisting counterparts (where the shears reinforce one another). Compared to previous studies we focus on the parameter space corresponding to weak shears, near the critical parameter below which hyperbolicity is lost and the map is non-mixing. The approach developed to deal with this situation appears applicable to a broad range of non-uniformly hyperbolic examples.
\end{abstract}

\Acknowledgements{JMH supported by EPSRC under Grant Ref. EP/W524372/1.}


\section{Introduction}
\label{sec:intro}

A well known example of non-uniform hyperbolicity, linked twist maps \cite{burton_ergodicity_1980,wojtkowski_linked_1980,przytycki_ergodicity_1983,sturman_mathematical_2006,springham_ergodic_2008} (hereafter LTMs) arise in a variety of applications. As models of chaotic advection in the presence of boundaries, their dynamics are central to problems in laminar mixing (see \cite{sturman_mathematical_2006} and the references therein) and other physical phenomena \cite{devaney_subshifts_1978,sivaramakrishnan_linked_1989}. More recently \cite{hasselblatt_hyperbolicity_2023} drew a connection between LTMs and certain contact flows \cite{foulon_orbit_2021}.

We adopt the general form of a linear toral LTM from \cite{springham_ergodic_2008}. Fix four constants $0\leq x_1<x_2<1$, $0 \leq y_1<y_2<1$ and parameterise the torus by $(x,y)\in S^1 \times S^1$. Let $f:[y_0,y_1] \to S^1$ and $g:[x_0,x_1] \to S^1$ be given by
$f(y) = (y-y_0)/(y_1-y_0)$ and $g(x) = (x-x_0)(x_1-x_0)$. Defining horizontal and vertical annuli $P = \{ (x,y) \, | \, y_0 \leq y \leq y_1 \}$ and $Q =\{ (x,y) \, | \, x_0 \leq x\leq x_1 \}$, union $R= P \cup Q$, let $F,\tilde{G}:R \to R$ be given by
\[F(x,y) = \begin{cases} (x+f(y),y) & (x,y) \in P, \\ (x,y) & \text{otherwise,} \end{cases} \quad \tilde{G}(x,y) = \begin{cases} (x,y+g(x)) & (x,y) \in Q, \\ (x,y) & \text{otherwise,} \end{cases} \]
with the coordinates calculated modulo 1. For integers $k,l \neq 0$, the composition $H_{k,l} = \tilde{G}^l \circ F^k$ forms a continuous piecewise-linear Lebesgue measure preserving transformation on $R$.

We focus on the case of \emph{counter-twisting} LTMs where $k$ and $l$ have opposite signs. Proving mixing properties of co-twisting LTMs (with $kl>0$) is more straightforward; see \cite{burton_ergodicity_1980,wojtkowski_linked_1980}. Defining constants $\alpha=kf' = k/(y_0-y_1)$ and $\beta = lg' = l/(x_1-x_0)$, \cite{przytycki_ergodicity_1983} showed that if $|k|,|l| \geq 2$ and $\alpha\beta<-C \approx -17.244$ then $H_{k,l}$ is Bernoulli. Recently \cite{pathak_ergodicity_2023} revisited the problem, removing the constraint $|k|,|l| \geq 2$ and proving mixing properties up to ergodicity for $\alpha\beta<-C \approx -12.04$. Here we continue this effort, focusing on the case of single twists $|k| = |l| = 1$ of weak strength $|\alpha|,|\beta| \approx 2$. As in \cite{przytycki_ergodicity_1983,pathak_ergodicity_2023} we rescale so that $|\alpha| = |\beta|$. Without loss of generality we take $k=1$, $l=-1$, and shift $R$ so that $x_0=y_0= 0$, giving $x_1=y_1=1/\alpha$. Writing $G = \tilde{G}^{-1}$, the map $H = G \circ F$ is then parameterised by a single positive parameter $\alpha$. A sketch is given in Figure \ref{fig:simpleLTM}. Our main theorem is as follows:

\begin{figure}
    \centering
 \begin{tikzpicture}
    \tikzmath{\a = 2.4;}
 \node at (-5.2,0) {    
\begin{tikzpicture}[scale=1.6]
\draw (0,0) -- (0,3) -- (3/\a,3) -- (3/\a,3/\a) -- (3,3/\a) -- (3,0) -- (0,0);
\draw (0,3/\a) -- (3/\a,3/\a) -- (3/\a,0);

\node at ({3/(2*\a)},{(3 + 3/\a)/2}) {$R \setminus P$};
\node at ({3/(2*\a)},{3/(2*\a)}) {$P \cap Q$};
\node at ({(3 + 3/\a)/2},{3/(2*\a)}) {$R \setminus Q$};
\draw[very thick,dashed] (3/\a,0) -- (3,0);
\draw[very thick,dashed] (3/\a,3/\a) -- (3,3/\a);

\draw[very thick,dashed] (3/\a,{3/(2*\a)}) -- (1.5,{3/(2*\a)});
\draw[very thick,dashed] (3,{3/(2*\a)}) -- (1.5+3/\a,{3/(2*\a)});

\node at ( { (3/\a +  1.5 )/2 } , {3/(2*\a) + 0.15} ) {$L_2^\star$};
\node at ( { (3 +  1.5+3/\a )/2 } , {3/(2*\a) + 0.15} ) {$L_2$};

\node at ({(3 + 3/\a)/2},0.15) {$L_1$};
\node at ({(3 + 3/\a)/2},3/\a +0.15) {$L_1^\star$};

\end{tikzpicture}};    
    
 \node at (0,0) {    
\begin{tikzpicture}[scale=1.6]
\draw (0,0) -- (0,3) -- (3/\a,3) -- (3/\a,3/\a) -- (3,3/\a) -- (3,0) -- (0,0);
\draw (0,3/\a) -- (3,3/\a) -- (0,0);
\foreach \n in {1,...,5}{
\draw[->] (0,{3*\n/(6*\a}) -- (3*\n/6-0.15,{3*\n/(6*\a});
}
\node[scale=2] at ({(3 + 3/\a)/2},{(3 + 3/\a)/2}) {$F$};
\end{tikzpicture}};

 \node at (5.2,0) { 
\begin{tikzpicture}[scale=1.6]
\draw (0,0) -- (0,3) -- (3/\a,3) -- (3/\a,3/\a) -- (3,3/\a) -- (3,0) -- (0,0);
\draw (0,3) -- (3/\a,0) -- (3/\a,3);

\foreach \n in {1,...,5}{
\draw[->] ({3*\n/(6*\a},3) -- ({3*\n/(6*\a},3-3*\n/6+0.15);
}
\node[scale=2] at ({(3 + 3/\a)/2},{(3 + 3/\a)/2}) {$G$};
\end{tikzpicture}};
 \end{tikzpicture}
    \caption{A linear counter-twisting toral linked twist map $H=G \circ F$ on the region $R = P \cup Q$. Dashed lines denote periodic segments in $R \setminus Q$; case illustrated $\alpha  = 2.4$.}
    \label{fig:simpleLTM}
\end{figure}
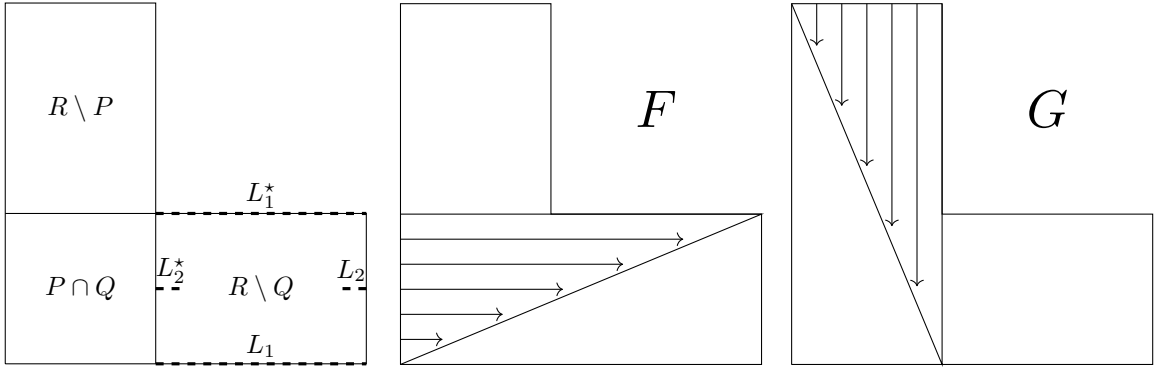

\begin{thm}
    \label{thm:Bernoulli}
    Let $3 > \alpha > \alpha_0 \approx 2.1319$. Over this parameter range $H$ has the Bernoulli property. 
\end{thm}

For comparison with \cite{przytycki_ergodicity_1983,pathak_ergodicity_2023}, Theorem \ref{thm:Bernoulli} covers the range $-4.545 \approx -\alpha_0^2 > \alpha\beta > -9$. This brings us close to, yet still bounded away from, the `optimal' shear parameter of $\alpha=2$, below which $H$ is non-ergodic (see e.g. \cite{sturman_mathematical_2006}, Figure 6.12). We claim that $\alpha_0$ is essentially the lowest achievable bound on the mixing window when relying on a single iterate of the canonical induced map for expansion. We discuss this further in section \ref{sec:remarks}, outlining the likely necessary method for dealing with the remaining parameter space.

The fundamental obstacle in the counter-twisting setting is outlined in \cite{przytycki_ergodicity_1983}. While hyperbolicity provides the expansion needed to prove mixing properties, it is tempered by the folding effect of the \emph{singularities}, where the Jacobian of $H$ or its higher powers are undefined. Showing that hyperbolicity dominates this interaction (establishing so called complexity estimates) is a key step in proving statistical properties of various chaotic systems, e.g. billiards. \cite{springham_polynomial_2014,myers_hill_exponential_2022,myers_hill_loss_2023} give detailed examinations of these estimates for maps very similar to $H$, using them to prove results on mixing rates using the schemes of \cite{chernov_billiards_2005,chernov_statistical_2009}.

In slow mixing systems such as LTMs, laborious calculations are often necessary to verify the estimates directly. Particularly, as is the case with $H$ for small $\alpha$, when hyperbolicity is weak and one must consider higher powers of the map. An indirect approach is often more practical, relying on other features of the map to simplify the problem. Whereas \cite{przytycki_ergodicity_1983} relies on the specific structure of periodic points under $F$, here we employ basic complexity estimates and exploit the self-similar structure of the induced map. This approach, developed in \cite{myers_hill_loss_2023}, appears applicable to a broad range of non-uniformly hyperbolic examples, whose induced return maps typically admit this self-similar structure. For example in \cite{myers_hill_mixing_2022} it is applied to a non-monotonic LTM, possessing both co-twisting and counter-twisting dynamics. The present work demonstrates the method for a well known map, over a broad parameter range. It is organised as follows. In section \ref{sec:background} we recall some classical results, in particular the scheme of \cite{katok_invariant_1986} used to show the Bernoulli property. Sections \ref{sec:returnTimes}, \ref{sec:lemmas} deal with the structure of the induced map and the images of certain line segments interacting with it. Section \ref{sec:growth} proves the growth lemma, the key step in the of proof Theorem \ref{thm:Bernoulli}, given in section \ref{sec:proof}. We conclude with some remarks on extension to the full expected mixing parameter range $\alpha \geq 2$.

\section{Background results}

\label{sec:background}
By \cite{wojtkowski_linked_1980} $H$ is \emph{hyperbolic}, possessing non-zero Lyapunov exponents almost everywhere, for all $\alpha >2$. This, together with mild conditions on the singularity set (see \cite{sturman_mathematical_2006} for a detailed treatment), implies the existence of local unstable and stable manifolds $\gamma_u(z)$ and $\gamma_s(z)$ at almost every $z \in R$. This result is due to \cite{katok_invariant_1986}, a generalisation of Pesin theory \cite{pesin_characteristic_1977} for `smooth maps with singularities'. The Bernoulli property (and by extension the lower rungs of the ergodic hierarchy: strong mixing, ergodicity etc.) follows from establishing:
\begin{enumerate}[label=(\textbf{MR}):]
    \item For almost any $z,\zeta \in R$, there exist $M,N$ such that for all $m\geq M$ and $n\geq N$, $H^m\gamma_u(z) \cap H^{-n}\gamma_s(\zeta) \neq \varnothing$.
\end{enumerate}
As $H$ is piecewise-linear and non-uniformly hyperbolic, local manifolds are line segments whose diameters may be arbitrarily small. A key step in showing (\textbf{MR}) is establishing exponential expansion in the diameters of $H^m\gamma_u(z)$ and  $H^{-n}\gamma_s(\zeta)$, growing these images up to some tangible size where intersections may be inferred. 

A uniformly hyperbolic induced map forms the basis for this expansion. The canonical choice is the return map $H_S: S \to S$, $z \mapsto H^{r}(z)$, where $r = r(z;H,S) = \min \{  i > 0 \, | \, H^i(z) \in S \} $ denotes the return time of $z$ to $S= P \cap Q$ under $H$. It decomposes as the composition $H_S = G_S \circ F_S$ of returns under $F$ then $G$ so $H_S(z) = G^l \circ F^k(z)$ for some naturals $k,l \geq 1$ depending on $z$. The return time of $z$ to $S$ is then given by $r = k + l -1$. Recall from \cite{przytycki_ergodicity_1983} the cone $\mathcal{C}$, defined by vectors $(v_1,v_2) \in \mathbb{R}^2 \setminus \{ \mathbf{0} \}$ with $ L \leq v_1/v_2 \leq 0$ where $L = \frac{1}{2} \left( -\alpha + \sqrt{\alpha^2 -4} \right)$. Similarly defining $\mathcal{C}'$ by $v_1/v_2 \geq L + \alpha$, one has $DF^k \mathcal{C} \subset \mathcal{C}'$ and $DG^l \mathcal{C}' \subset \mathcal{C}$ for all $k,l \geq 1$. Hence $\mathcal{C}$ is invariant under the Jacobian $DH_S$; it further provides bounds on the gradients of local manifolds mapped into $S$ by $H$. In particular at almost every $z \in R$, we can find $i>0$ such that $H^i \gamma_u(z)$ contains a linear segment $\Gamma_0 \subset S$, aligned with some $v \in \mathcal{C}$ \cite{przytycki_ergodicity_1983}. The stable cone $\mathcal{C}^s$, defined by the inequality $ L \leq v_2/v_1 \leq 0$, similarly bounds the backwards images of stable manifolds $H^{-i}\gamma_s(z) \subset S$.

\section{Structure of return times}
\label{sec:returnTimes}
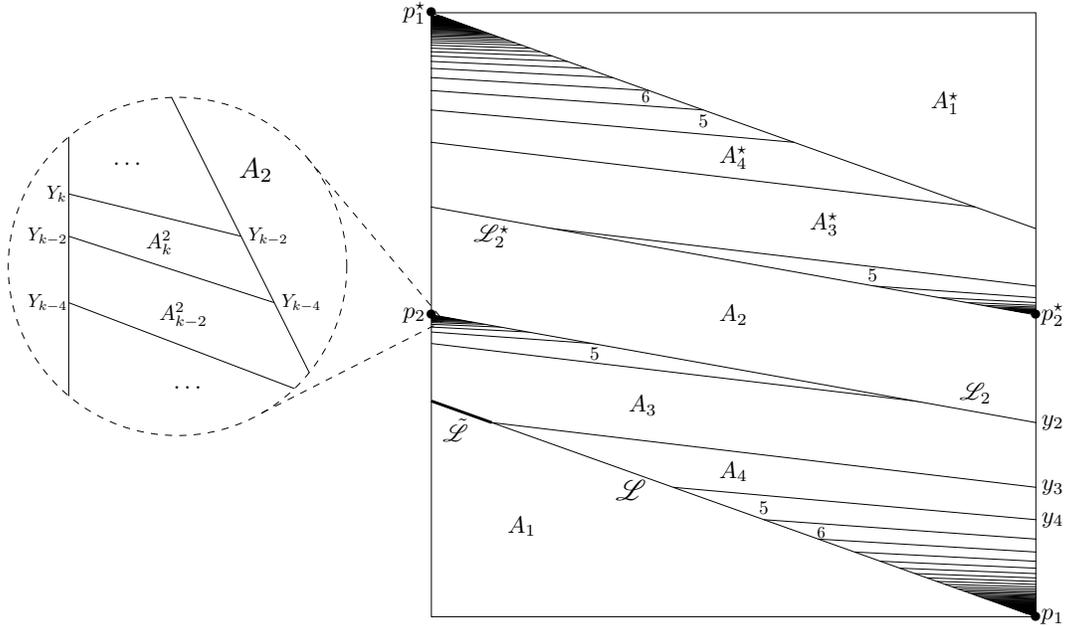
\begin{figure}
    \centering

\begin{tikzpicture}

\node[scale = 0.9] at (0,0) {

\begin{tikzpicture}[scale = 2.5]
    \tikzmath{\a = 2.8;}

    \draw (0,0) rectangle (10/\a,10/\a);

    \node at (10/\a,0) {$\bullet$};
    \node at (0,10/\a) {$\bullet$};
    \node at (0,5/\a) {$\bullet$};
    \node at (10/\a,5/\a) {$\bullet$};

    \node at (10/\a+0.1,0) {$p_1$};
    \node at (-0.1,10/\a) {$p_1^\star$};
    \node at (-0.1,5/\a) {$p_2$};
    \node at (10/\a+0.1,5/\a) {$p_2^\star$};

    \node at  (10/\a+0.1, { 10*(1/\a - 1/\a^2)/2 }) {$y_2$};
    \node at  (10/\a+0.1, { 10*(1/\a - 1/\a^2)/3 }) {$y_3$};
    \node at  (10/\a+0.1, { 10*(1/\a - 1/\a^2)/4 }) {$y_4$};

    \begin{scope}

    \clip (0,0) rectangle (10/\a,10/\a);

    \node at (5/\a,5/\a) {$A_2$};
    \node at (1.5/\a,1.5/\a) {$A_1$};
    \node at (8.5/\a,8.5/\a) {$A_1^\star$};
    \node at (3.5/\a,3.5/\a) {$A_3$};
    \node at (6.5/\a,6.5/\a) {$A_3^\star$};
    \node at (5/\a,2.4/\a) {$A_4$};
    \node[scale=0.8] at (2.7/\a,4.35/\a) {5};
    \node at (5/\a,7.6/\a) {$A_4^\star$};
    \node[scale=0.8] at (7.3/\a,5.65/\a) {5};
    \node[scale=0.8] at (5.5/\a,1.8/\a) {5};
    \node[scale=0.7] at (6.45/\a,1.4/\a) {6};

        \node[scale=0.8] at (4.5/\a,8.2/\a) {5};
    \node[scale=0.7] at (3.55/\a,8.6/\a) {6};

    \node[scale=1.2] at  (3.3/\a,2.1/\a) {$\mathscr{L}$};

    \node[scale=1] at  (9/\a,3.7/\a) {$\mathscr{L}_2$};
    \node[scale=1] at  (1/\a,6.3/\a) {$\mathscr{L}_2^\star$};

    \draw[very thick] (0,{10/\a^2}) -- ({10*( 3/\a - 1)/2},{ 10*(1/\a - 1/\a^2)/2  });
    \node at (0.4/\a, {10/\a^2 - 0.17}) {$\tilde{\mathscr{L}}$};
    
    \draw (0, {10/(\a*\a)  }) -- (10/\a,0);

    \foreach \k in {2,...,120}{
    \draw ( {10*(\k -\a)/( \a*\k - \a  ) }  ,{ 10*(\a -1)/( \a*\a*\k - \a*\a ) } ) -- (10,0);

    }

    \begin{scope}
    \clip (10,0) -- ( {10*(2 -\a)/( \a*2 - \a  ) }  ,{ 10*(\a -1)/( \a*\a*2 - \a*\a ) } ) -- (0,0) -- (10,0);
    \foreach \n in {1,...,20}{
    \draw ( 0 , {10*(\n + 1/\a)/( (2*\n+1 )*\a )}) -- ( 10/\a , {10*\n/((2*\n+1)*\a ) } );
    }
    
    \end{scope}

    \node at (0,{(5*\a + 5)/(\a*\a)}) {

    \begin{tikzpicture}[scale = 2.5,rotate=180]
        \draw (0, {10/(\a*\a)  }) -- (10/\a,0);

    \foreach \k in {2,...,120}{
    \draw ( {10*(\k -\a)/( \a*\k - \a  ) }  ,{ 10*(\a -1)/( \a*\a*\k - \a*\a ) } ) -- (10,0);

    }

    \begin{scope}
    \clip (10,0) -- ( {10*(2 -\a)/( \a*2 - \a  ) }  ,{ 10*(\a -1)/( \a*\a*2 - \a*\a ) } ) -- (0,0) -- (10,0);
    \foreach \n in {1,...,20}{
    \draw ( 0 , {10*(\n + 1/\a)/( (2*\n+1 )*\a )}) -- ( 10/\a , {10*\n/((2*\n+1)*\a ) } );
    }
    
    \end{scope}

    \end{tikzpicture}
    
    };

  \end{scope}

    \draw[dashed] (-4.2/\a,5.8/\a) circle (2.8/\a);
    \draw[dashed] ({-4.2/\a + 2.8*cos(300)/\a },{5.8/\a + 2.8*sin(300)/\a }) -- (0.2/\a,4.9/\a);
    \draw[dashed] ({-4.2/\a + 2.8*cos(40)/\a },{5.8/\a + 2.8*sin(40)/\a }) -- (0.2/\a,4.9/\a);
    \clip (-4.2/\a,5.8/\a) circle (2.8/\a);
    \draw (-6/\a,10/\a) -- (-6/\a,0);
    \draw (-6/\a,12/\a) -- (0,0);

    \draw (-6/\a,7/\a) -- (-3.15/\a,6.3/\a);
    \draw (-6/\a,6.3/\a) -- (-2.6/\a,5.2/\a);
    \draw (-6/\a,5.2/\a) -- (-1.8/\a,3.6/\a);

 \node[scale=0.8] at (-6.2/\a,7/\a) {$Y_k$};
    \node[scale=0.8] at (-6.35/\a,6.3/\a) {$Y_{k-2}$};
    \node[scale=0.8] at (-6.35/\a,5.2/\a) {$Y_{k-4}$};

    \node[scale=0.8] at (-2.7/\a,6.3/\a) {$Y_{k-2}$};
    \node[scale=0.8] at (-2.15/\a,5.2/\a) {$Y_{k-4}$};

    \node at (-4/\a,3.8/\a) {$\dots$};
    \node at (-5/\a,7.5/\a) {$\dots$};
    \node[scale=0.9] at (-4.5/\a, 6.2/\a) {$A_k^2$};
    \node[scale=0.9] at (-4.1/\a, 5/\a) {$A_{k-2}^2$};
    \node[scale=1.2] at (-2.9/\a, 7.4/\a) {$A_2$};
\end{tikzpicture}

    };
\end{tikzpicture}
    
    \caption{Partition of $S$ into sets $A_k^{(\dots)}$ of return time $k$ under $F$. Each are bounded by preimages of $\partial S$; for example near $p_1$ each $A_k$ is bounded between $\mathscr{L}_{k-1}$ and $\mathscr{L}_k$, meeting $\partial S$ at $(1/\alpha, y_{k-1})$ and $(1/\alpha, y_k)$ respectively. Near $p_2$ the accumulating sets have odd return times, bounded between $\mathscr{L}_{k-2}^2$ and $\mathscr{L}_k^2$ which meet $\partial S$ at $(0,Y_{k-2})$, $(0,Y_k)$ respectively. Case illustrated $\alpha = 2.8$.}
    \label{fig:Fsing}
\end{figure}

We begin by describing the structure of hitting times $h(z;F,S) = \min \{  i > 0 \, | \, F^i(z) \in S \}$ over $z \in P$. For $2 < \alpha < 3$ such integers exist outside of the four line segments in $P \setminus S$:
\begin{itemize}
    \item $L_1$: $y=0$, $\frac{1}{\alpha}< x < 1$,
    \item $L_1^\star$: $y=1$, $\frac{1}{\alpha}< x < 1$,
    \item $L_2$: $y = \frac{1}{2\alpha}$, $\frac{1}{2} + \frac{1}{\alpha} < x <1$,
    \item $L_2^\star$: $y = \frac{1}{2\alpha}$, $\frac{1}{\alpha} < x < \frac{1}{2}$,
\end{itemize}
each periodic under $F$ with period given by the subscript. These segments are sketched as the dashed lines in Figure \ref{fig:simpleLTM}. Across the rest of the parent circles $y \in \{0, \frac{1}{2\alpha}, \frac{1}{\alpha} \}$, points hit $S$ in just one or two iterates. The structure of return times to $S$ is plotted in Figure \ref{fig:Fsing}; the subscripts of the labelled regions correspond to the return time $r(\cdot;F,S)$. Near the circles $y \in \{0, \frac{1}{2\alpha}, \frac{1}{\alpha} \}$ we see either very fast returns, $r \in \{ 1,2\}$, or very slow returns, with $r$ diverging as we approach the accumulation points $p_i^{(\star)}$. In particular each $p_i^{(\star)}$ lies in the closure of $L_i^{(\star)}$ on $\partial S$; we label the four segments $\partial S_j$ which make up this boundary as shown in Figure \ref{fig:segments}. For example near $p_1$ we have the large immediately returning set $A_1 \subset F^{-1}(S) \cap S$ of points which shift no further than $\partial S_2$ under $F$, and a self similar family of sets $A_k$. For $k \geq 4$, the orbit of a point $(x,y) \in A_k$ maps into $P \setminus S$ under $F$ then shifts horizontally by $k-1$ further increments of $\alpha y$, hitting $S$ just beyond the left boundary $\partial S_1$. Each $A_k$ is then bounded by the segment $\mathscr{L} \subset F^{-1}(\partial S_2)$, $\partial S_2$, $\mathscr{L}_{k-1}$ and $\mathscr{L}_k$, where $\mathscr{L}_k \subset F^{-k}(\partial S_1)$. For future reference, $\mathscr{L}$ lies on the line
\begin{equation}
    \label{eq:L}
    y = -\frac{1}{\alpha} \left( x-\frac{1}{\alpha} \right)
\end{equation}
and the $\mathscr{L}_k$ lie on the lines
\begin{equation}
\label{eq:Lk}
    y =  -\frac{1}{k\alpha} \left(x -1 \right).
\end{equation}

Near $p_2$, points $(x,y)$ either return under $F^2$ (the set $A_2$, bounded by the segment $\mathscr{L}_2 \subset F^{-2}(\partial S_1) $) or $F(x,y)$ falls just short of the line $x=1/2$. Letting $(x_n,y) = F^{2n+1}(x,y)$, the sequence $x_n = x_{n-1} + 2\alpha y \text{ mod }1 $ is strictly decreasing, giving $N$ such that $x_N < 1/\alpha$, so that $F^{2N+1}(x,y)$ lies left of $\partial S_2$. This gives odd first return times to $S$ as the even iterates up to $n=2N$ lie further right near the segment $L_2$. The number $N$ diverges as we approach $p_2$ ($y \to \frac{1}{2\alpha}$) giving, for $k \geq 5$, secondary accumulating sets of constant return time $A_k^2$ bounded by $\mathscr{L}_2$, $\partial S_1$, $\mathscr{L}_{k-2}^2$ and $\mathscr{L}_k^2$ where $\mathscr{L}_k^2 \subset F^{-k}(\partial S_2)$. For future reference, the $\mathscr{L}_k^2$ lie on the lines
\begin{equation}
    \label{eq:Lk2}
    y = \frac{1}{k\alpha} \left( \frac{k-1}{2} + \frac{1}{\alpha} - x \right).
\end{equation}

The remaining region between these two accumulating patterns forms the set $A_3$ of constant return time 3, completing the description of return times to $S$ over $y < \frac{1}{2\alpha}$. Noting that $F$ commutes with the involution $I_1(x,y) = \left( \frac{1}{\alpha} - x ,\frac{1}{\alpha} - y \right) \text{ mod 1}$, we may infer sets of constant return time above $y = \frac{1}{2\alpha}$ by mapping under $I_1$. Figure \ref{fig:Fsing} provides a plot at an example parameter, denoting the images under $I_1$ with a superscript $\star$.

Defining three further transformations
$I_2(x,y) = \left( x, \frac{1}{\alpha} - y \right) \text{ mod 1}$, $I_3(x,y) = \left( y , \frac{1}{\alpha} - x \right) \text{ mod 1}$, $I_4(x,y) = \left( y, x \right) $, and the map $\mathcal{H} = F \circ G$, the following relations are straightforward to verify:
\begin{fact}
    \begin{enumerate}[label=(\alph*)]
        \item $I_1$ commutes with $F$, $G$ and by extension $H$.
        \item $I_2 \circ G = G^{-1} \circ I_2$.
        \item $I_3 \circ F = G \circ I_3$.
        \item $I_3 \circ H = \mathcal{H} \circ I_3$.
        \item $I_4 \circ H = H^{-1} \circ I_4$.
    \end{enumerate}
    The same relations hold for all powers of the maps $F,G,H$, and by extension the return maps $F_S,G_S,H_S$.
\end{fact}

The transformation $I_3$ allows us to deduce the structure of $G_S$ from that of $F_S$, sketched in Figure \ref{fig:Gsing}. The labelling scheme follows directly from Figure \ref{fig:Fsing} with e.g. the set $B_4$ having return time 4 under $G$, bounded by the segments $\mathscr{I} = I_3(\mathscr{L})$, $\mathscr{I}_3 = I_3(\mathscr{L}_3)$, $\mathscr{I}_4 = I_3(\mathscr{L}_4)$, $\partial S_3 = I_3(\partial S_2)$. For future reference $\mathscr{I}$ lies on the line $y= \alpha x$ and the $\mathscr{I}_k$ lie on the lines
\begin{equation}
    \label{eq:Ik}
    y = k \alpha x - 1 +\frac{1}{\alpha}.
\end{equation}

\begin{figure}
    \centering

\subfigure[][]{%
\label{fig:Gsing}%
   \begin{tikzpicture}
    \node[rotate=-90] at (0,0) {

\begin{tikzpicture}[scale = 2]
    \tikzmath{\a = 2.8;}

    \draw (0,0) rectangle (10/\a,10/\a);

    \node[rotate=90] at (10/\a,0) {$\bullet$};
    \node[rotate=90] at (0,10/\a) {$\bullet$};
    \node[rotate=90] at (0,5/\a) {$\bullet$};
    \node[rotate=90] at (10/\a,5/\a) {$\bullet$};

    \node[rotate=90] at (10/\a+0.15,0) {$q_1$};
    \node[rotate=90] at (-0.15,10/\a) {$q_1^\star$};
    \node[rotate=90] at (-0.15,5/\a) {$q_2$};
    \node[rotate=90] at (10/\a+0.15,5/\a) {$q_2^\star$};

    \begin{scope}

    \clip (0,0) rectangle (10/\a,10/\a);

    \node[rotate=90] at (5/\a,5/\a) {$B_2$};
    \node[rotate=90] at (1.5/\a,1.5/\a) {$B_1$};
    \node[rotate=90] at (8.5/\a,8.5/\a) {$B_1^\star$};
    \node[rotate=90] at (3.5/\a,3.5/\a) {$B_3$};
    \node[rotate=90] at (6.5/\a,6.5/\a) {$B_3^\star$};
    \node[rotate=90] at (5/\a,2.4/\a) {$B_4$};
    \node[rotate=90] at (5/\a,7.6/\a) {$B_4^\star$};

    \node[rotate=90][scale=1.2] at  (3.3/\a,2.1/\a) {$\mathscr{I}$};

    \node[rotate=90] at  (9/\a,3.7/\a) {$\mathscr{I}_2$};
    \node[rotate=90] at  (1/\a,6.3/\a) {$\mathscr{I}_2^\star$};
    
    \draw (0, {10/(\a*\a)  }) -- (10/\a,0);

    \foreach \k in {2,...,120}{
    \draw ( {10*(\k -\a)/( \a*\k - \a  ) }  ,{ 10*(\a -1)/( \a*\a*\k - \a*\a ) } ) -- (10,0);

    }

    \begin{scope}
    \clip (10,0) -- ( {10*(2 -\a)/( \a*2 - \a  ) }  ,{ 10*(\a -1)/( \a*\a*2 - \a*\a ) } ) -- (0,0) -- (10,0);
    \foreach \n in {1,...,20}{
    \draw ( 0 , {10*(\n + 1/\a)/( (2*\n+1 )*\a )}) -- ( 10/\a , {10*\n/((2*\n+1)*\a ) } );
    }
    
    \end{scope}

    \node at (0,{(5*\a + 5)/(\a*\a)}) {

    \begin{tikzpicture}[scale = 2,rotate=180]
        \draw (0, {10/(\a*\a)  }) -- (10/\a,0);

    \foreach \k in {2,...,120}{
    \draw ( {10*(\k -\a)/( \a*\k - \a  ) }  ,{ 10*(\a -1)/( \a*\a*\k - \a*\a ) } ) -- (10,0);

    }

    \begin{scope}
    \clip (10,0) -- ( {10*(2 -\a)/( \a*2 - \a  ) }  ,{ 10*(\a -1)/( \a*\a*2 - \a*\a ) } ) -- (0,0) -- (10,0);
    \foreach \n in {1,...,20}{
    \draw ( 0 , {10*(\n + 1/\a)/( (2*\n+1 )*\a )}) -- ( 10/\a , {10*\n/((2*\n+1)*\a ) } );
    }
    
    \end{scope}

    \end{tikzpicture}
    
    };

  \end{scope}
\end{tikzpicture}

    };
\end{tikzpicture}
    }%
    \subfigure[][]{%
\label{fig:segments}%
    \begin{tikzpicture}[scale = 2]
    \tikzmath{\a = 2.8;}
    
    \filldraw[fill=gray!50] (5/\a, 0) -- ({ 10/\a - 10*(1/\a + 1/\a^2)/2 },10/\a) -- (5/\a, 10/\a) -- ({ 10*(1/\a + 1/\a^2)/2 },0) -- (5/\a, 0);
    
    \draw[very thick] (0,0) rectangle (10/\a,10/\a);
    \node[white] at (5/\a, -0.21) {$q_2^\star$};
    \node[scale=1.2] at (0.2,6.5/\a) {$\partial S_1$};
    \node[scale=1.2] at (10/\a - 0.2,6.5/\a) {$\partial S_2$};
    \node[scale=1.2] at (5/\a,-0.15) {$\partial S_3$};
    \node[scale=1.2] at (5/\a, 10/\a + 0.15) {$\partial S_4$};

    \draw[very thick] ( 6.4/\a, 0 ) -- (4/\a,10/\a);

    \draw (5/\a, 10/\a) -- ({ 10*(1/\a + 1/\a^2)/2 },0);
    \draw (5/\a, 0) -- ({ 10/\a - 10*(1/\a + 1/\a^2)/2 },10/\a);

    \node[scale=0.9] at (2.8/\a,7.5/\a) {$I_2(\mathscr{I}_2^\star)$};
    \node[scale=0.9] at (6.2/\a,7.5/\a) {$I_2(\mathscr{I}_2)$};

    \node[scale=1.2] at (5.6/\a,1/\a) {$\Gamma''$};

    \draw  ({10*( 3/\a - 1)/2},{ 10*(1/\a - 1/\a^2)/2  }) -- (10/\a, {10*(1/\a - 1/\a^2)/3});
    \draw  (0,5/\a) -- (10/\a, {10*(1/\a - 1/\a^2)/2});

    \node at ({10*( 3/\a - 1)/2},{ 10*(1/\a - 1/\a^2)/2  }) {$\bullet$};
    \node[scale=0.9] at ({10*( 3/\a - 1)/2},{ 10*(1/\a - 1/\a^2)/2 - 0.15  }) {$\mathscr{L}_3 \cap \mathscr{L}$};

    \node at (9/\a,3.7/\a) {$\mathscr{L}_2$};
    \node at (9/\a,1.85/\a) {$\mathscr{L}_3$};

    \end{tikzpicture}
    }%
    \caption{Part (a) shows the distribution of return times to $S$ under $G$, using a similar labelling scheme to Figure \ref{fig:Fsing}. Part (b) gives a labelling of $\partial S$ and sketches a $v$-segment $\Gamma'' \subset G^2(B_2)$ (shaded).}
    \label{fig:other}
\end{figure}
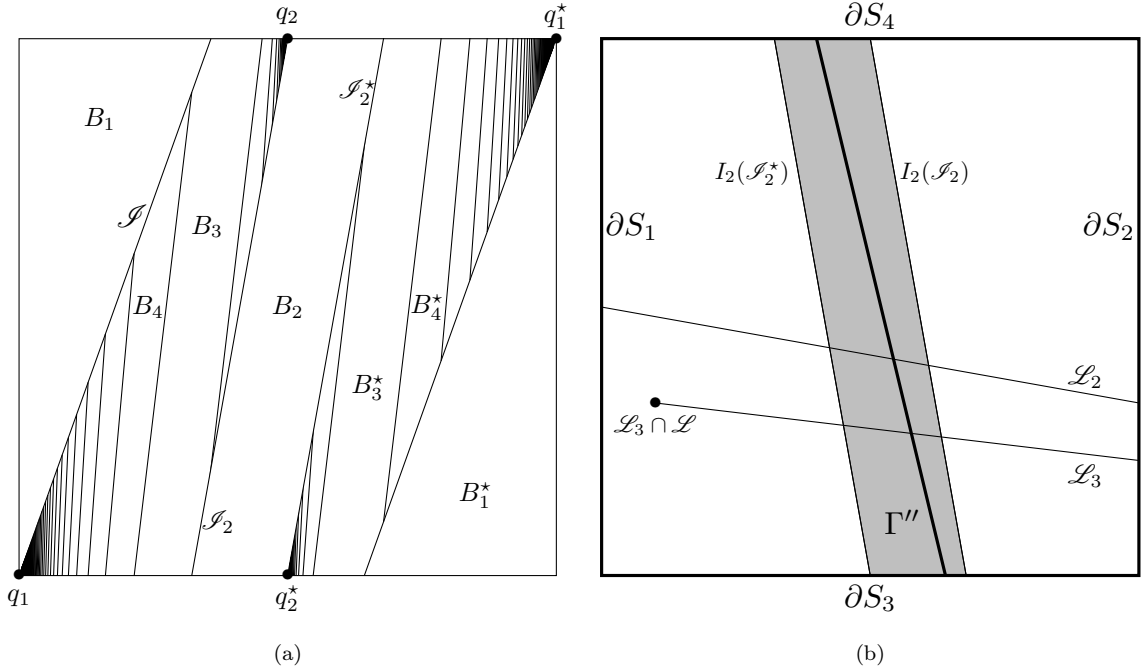

\section{Mapping into $v$-segments}
\label{sec:lemmas}
Given a line segment $\Gamma \subset S$, we say that $\Gamma$ is a $h$-segment if it connects $\partial S_1$ to $\partial S_2$. Similarly we call $\Gamma$ a $v$-segment if it connects $\partial S_3$ to $\partial S_4$ (see Figure \ref{fig:segments} for an example). We begin by showing how particular line segments map into $v$-segments.

\begin{lemma}
    \label{lemma:Lkmapping}
    Let $\alpha \geq \alpha_1 \approx 2.125$, the largest root of the cubic equation $2\alpha^3 -4 \alpha^2 - \alpha +1 =0$. Given a line segment $\Gamma \subset S$ aligned with some $v \in \mathcal{C}$:
    \begin{enumerate}
        \item If $\Gamma$ connects $\mathscr{L}_{k-1}$ to $\mathscr{L}_k$ for some $k \geq 4$, $F_S(\Gamma)$ contains a segment connecting $\mathscr{I}_{k-2}$ to $\mathscr{I}_{k-1}$.
        \item If $\Gamma$ connects $\mathscr{L}_2$ to $\mathscr{L}_3$, $H^4(\Gamma)$ contains a $v$-segment.
    \end{enumerate}
\end{lemma}

\begin{proof}
    Starting with the first statement, restrict $\Gamma$ to $A_k$ and denote its endpoints by $z_k \in \mathscr{L}_k$ and $z_{k-1} \in \mathscr{L}_{k-1}$. We note that the $y$-coordinate of $z_{k-1}$ is bounded below by that of $\mathscr{L}_{k-1} \cap \partial S_2$, which by (\ref{eq:Lk}) is
    \begin{equation}
        \label{eq:ybound}
        y \geq y_{k-1} := \frac{1}{k-1} \left( \frac{1}{\alpha} - \frac{1}{\alpha^2} \right),
    \end{equation}
    and above by that of $\mathscr{L}_{k-1} \cap \mathscr{L}$, equal to $y_{k-2}$.
    With $\Gamma \subset A_k$ we have that $F_S(\Gamma) = F^k(\Gamma)$ and by definition $F^k(\mathscr{L}_k) \subset \partial S_1$ and $F^k(\mathscr{L}_{k-1}) \subset F(\partial S_1)$. The segment $F_S(\Gamma)$ thus connects $F^k(z_k) \in \partial S_1$ to $(\alpha y, y) : = F^k(z_{k-1}) \in F(\partial S_1)$. Given that $F_S(\Gamma)$ has non-negative gradient (aligned with $DF^kv \in \mathcal{C}'$), for $F_S(\Gamma)$ to intersect $\mathscr{I}_{k-2}$ and $\mathscr{I}_{k-1}$ it is sufficient to show that $(\alpha y, y)$ lies right of the line segment $\mathscr{I}_{k-2}$, with $y$ bounded above by the $y$-coordinate of $\mathscr{I}_{k-1} \cap \mathscr{I}$. Using the inequality $y_{k-1} \leq y \leq y_{k-2}$ and (\ref{eq:Ik}), this holds over the given parameter range for all $k \geq 4$ as required.

    The second statement follows by a similar argument. If $(\alpha y, y)$ lies right of the segment $\mathscr{I}_2^\star$ then $F^3(\Gamma \cap A_3) \subset F_S(\Gamma)$ contains a segment $\Gamma' \subset B_2$, connecting $\mathscr{I}_2$ to $\mathscr{I}_2^\star$. The image of such a segment under $G_S = G^2$ is a $v$-segment, so that $H^4(\Gamma)$ contains a $v$-segment. The segment $\mathscr{I}_{2}^\star$ lies on the line $y = 2\alpha x - 1$; since $y \geq y_2$ it is sufficient to verify:
    \begin{equation}
    \label{eq:2inequality}
        y_2 \leq  2 \alpha^2 y_2 - 1
    \end{equation}
    which reduces to
    \[ 2\alpha^3 - 4 \alpha^2 - \alpha + 1 \geq 0,\]
    valid for all $\alpha \geq \alpha_1$ as required.
\end{proof}

Consider the point $z_p = (x_p,y_p) \in A_3$, where
\begin{equation}
    \label{eq:P4point}
    (x_p,y_p) = \left(  \frac{2\alpha-4}{3\alpha^3 -8\alpha} ,\frac{\alpha^2 + \alpha -4}{3\alpha^3 -8\alpha} \right).
\end{equation}
It is periodic, of period 4 under $H$ with $F^3(z_p) \in B_1^\star$, $GF^3(z_p) \in A_1^\star$, $FGF^3(z_p) \in B_1$, giving
\[ DH_{z_p}^4 = DG \, DF \, DG \, DF^3 = \begin{pmatrix} -\alpha^2 + 1 &  3\alpha^3 + 4\alpha \\ \alpha^3-2 \alpha & 3\alpha^4 - 7\alpha^2 + 1\end{pmatrix}.  \]
For $\alpha > \sqrt{8/5} \approx 1.633$ this matrix is hyperbolic, possessing expanding and contracting eigenvectors $(1,g_+)^T$ and $(1,g_-)^T$ where
\[ g_{\pm} = \frac{4-2\alpha^{2}}{3\alpha^{3}-6\alpha \mp \sqrt{9\alpha^{6}-48\alpha^{4}+76\alpha^{2}-32}}.\] 
The region $M = F^{-3}( G^{-1} (  F^{-1} (  B_1 ) \cap A_1^\star ) \cap B_1^\star ) \cap A_3$ of points $z$ around $z_p$ similarly satisfying $DH_z^4 = DH_{z_p}^4$ is shaded in Figure \ref{fig:inclination}, bounded by $\partial A_3$ and the preimages $\mathscr{M}_1 \subset (FGF^3)^{-1}(\partial S_1)$, $\mathscr{M}_2 = (FGF^3)^{-1}(\mathscr{I})$. The line segment passing through $z_p$ with gradient $g_-$ and endpoints on $\partial M$ forms the stable manifold $\gamma_s$ at $z_p$. Defining the \emph{relative interior} of a line segment $\Gamma$ with endpoints $z_1,z_2$ as $\Gamma^\circ = \Gamma \setminus \{ z_1, z_2 \}$, we have the following:

\begin{lemma}
    \label{lemma:periodicPoint}
    Let $\alpha > \alpha_2 \approx 2.127$. Let $\Gamma$ be a line segment, aligned with $v \in \mathcal{C}$, which intersects $\gamma_s$ at some point $z_0 \in \Gamma^\circ$. Then there exists $k$ such that $H^k(\Gamma)$ contains a $v$-segment.
\end{lemma}

\begin{proof}
    We essentially apply the inclination or $\lambda$-lemma. Let $\Gamma_0 = \Gamma \cap M$ and iteratively define $\Gamma_i := H^4(\Gamma_{i-1} \cap M)$. This generates sequence of line segments $\Gamma_i \subset H^{4i}(\Gamma)$, aligned with $v_i := DH_{z_p}^{4i} v$, which pass through $z_i := H^{4i}(z_0) \in \gamma_s$. In effect, $\Gamma_i$ limits exponentially fast onto the unstable manifold $\gamma_u$ through $z_p$, with gradient $g_+$ and endpoints on $H^4(\mathscr{M}_1) \subset \partial S_1$ and $H^4(\mathscr{M}_2) \subset \partial S_3$ (plotted as the dashed line in Figure \ref{fig:inclination}). 
    
    For all $\alpha > \alpha_2$ we claim that either $F_S(\gamma_u)$ contains a segment $\gamma_u'$ whose relative interior intersects 
    $\mathscr{I}_2$ and $\mathscr{I}_2^\star$, or $F_S \circ H_S(\gamma_u)$ satisfies this intersection property. The lemma then follows, noting we can find finite $i$ such that $F_S(\Gamma_i)$ or $F_S \circ H_S(\Gamma_i)$ similarly intersects $\mathscr{I}_2$ and $\mathscr{I}_2^\star$, so that $H_S(\Gamma_i)$ or $H_S^2(\Gamma_i)$ contains a $v$-segment. In particular, letting $j$ denote the $j$th image $H^j(\Gamma_i)$ of $\Gamma_i$ containing this $v$-segment, $k$ in the lemma statement is given by $k=4i+j$.

    For all $\alpha>2$ the segment $\gamma_u$ intersects $\mathscr{L}_3^2$ and $\mathscr{L}$; write this latter intersection as $(x_u,y_u)$. If $\alpha > \alpha_3 \approx 2.694$, the parameter value for which $y_u = y_2$, the manifold $\gamma_u$ also intersects $\mathscr{L}_3$. The image $\gamma_u' = F^3(\gamma_u \cap A_3) \subset F_S(\gamma_u)$ is then a $h$-segment, intersecting $\mathscr{I}_2$ and $\mathscr{I}_2^\star$ in the desired fashion. Otherwise $\gamma_u'$ connects $(x_u',y_u)$ to $\partial S_2$, where $x_u' = 2\alpha(y-y_u)$ (as the image $F^3(\mathscr{L})$ has gradient $\frac{1}{2\alpha}$). The segment $\gamma_u'$ then intersects $\mathscr{I}^\star$ (with parent line $y= 1/\alpha + \alpha x -1$) provided that
    \[  y_u \geq 1/\alpha + \alpha x_u' -1 . \]
    This holds for $\alpha \geq \alpha_4 \approx 2.1239$, the precise parameter for which $(x_u',y_u) \in \mathscr{I}^\star$. Writing $(x,y) = \gamma_u' \cap \mathscr{I}^\star$,
    \[ x =  \frac{y - \frac{1}{\alpha} + 1}{\alpha},\]
    by cone alignment $\gamma_u'$ has non-negative gradient so $y \geq y_u$. The segment $\mathscr{I}^\star$ maps into $\partial S_4$ under $G$, so $\gamma_u'' = G(\gamma_u' \cap B_1^\star) \subset H_S(\gamma_u)$ joins $\partial S_2$ to $(x, 1/\alpha) \in \partial S_4$. In particular this endpoint lies to the right of $\mathscr{I}_2^\star$ provided
    \[ \frac{y - \frac{1}{\alpha} + 1}{\alpha} > \frac{1}{2}\left( \frac{1}{\alpha} + \frac{1}{\alpha^2} \right), \]
    the $x$-coordinate of $\mathscr{I}_2^\star \cap \partial S_4$. By $y \geq y_u$ it is sufficient to check this bound for $y=y_u$. Indeed it holds for $\alpha > \alpha_2 \approx 2.127$. As it lies in $F^3(\mathscr{L}_3^2)$, the other endpoint of $\gamma_u''$ on $\partial S_2$ lies below $\mathscr{L}^\star \cap \partial S_2$ so that $\gamma_u''$ intersects $\mathscr{L}^\star$. The image $F(\gamma_u'' \cap A_1^\star) \subset F_S\circ H_S(\gamma_u)$ then connects $\partial S_1$ to $(x , 1/\alpha)$, intersecting $\mathscr{I}_2$ and $\mathscr{I}_2^\star$ in the desired fashion.
\end{proof}

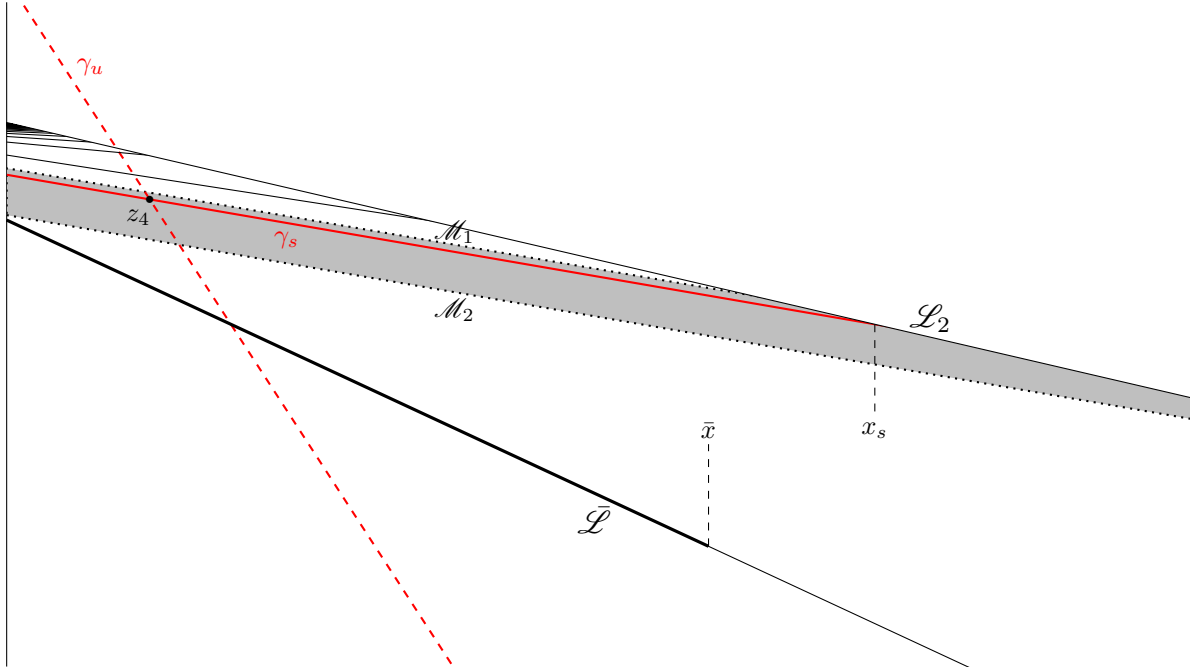
\begin{figure}
    \centering

    \begin{tikzpicture}[scale = 8]
    \tikzmath{\a = 2.15;}

    \draw (0,  {10/(2*\a) + 0.2 }) -- (0, { 10*(\a-1)/(2*\a*\a - \a) -0.2}   );
    
    \begin{scope}
    \clip (0, { 10*(\a-1)/(2*\a*\a - \a) -0.2}   ) rectangle ( { 10*(-\a*\a + 3*\a -1)/(2*\a*\a - \a) +0.8}   , {10/(2*\a) + 0.2 });

    \begin{scope}
        \clip (0,5/\a) -- (0,{10/(\a^2)}) -- ( {10*(3/\a - 1)/2} , {10*(1/\a - 1/(\a^2))/2}  ) -- (10/\a, {10*(1/\a - 1/(\a^2))/3}) -- (10/\a, {10*(1/\a - 1/(\a^2))/2}) -- (0,5/\a);
        \filldraw[thick, dotted, fill = gray!50] (  0, { 10*(\a^3 + \a^2 -3*\a -1)/( 1-7*\a^2 + 3*\a^4 )  }  ) -- (  10, { 10*(\a^2 -\a - 1)/( 1-7*\a^2 + 3*\a^4 )  }  ) -- (  10, { 10*(1-\a)/(-3*\a^3 + 4*\a )  }  ) --   (  0, { 10*(1-\a)*(\a+2)/(-3*\a^3 + 4*\a )  }  ) -- (  0, { 10*(\a^3 + \a^2 -3*\a -1)/( 1-7*\a^2 + 3*\a^4 )  }  );
        \draw[ red, domain=0:{4/\a}, smooth, variable=\x,  thick] plot ({\x}, { 10*(\a^2 + \a -4)/( 3*\a^3 - 8*\a) + (   (4-2*\a^2)/(3*\a^3-6*\a + (9*\a^6-48*\a^4+76*\a^2-32 )^0.5 ) )*(\x - 10*(2*\a - 4)/( 3*\a^3 - 8*\a))    });
    \end{scope}
     \draw[red, domain=0:{2/\a}, smooth, variable=\x, dashed,  thick] plot ({\x}, { 10*(\a^2 + \a -4)/( 3*\a^3 - 8*\a) + (   (4-2*\a^2)/(3*\a^3-6*\a - (9*\a^6-48*\a^4+76*\a^2-32 )^0.5 ) )*(\x - 10*(2*\a - 4)/( 3*\a^3 - 8*\a))    });

    \node[scale=0.7] at ({ 10*(2*\a - 4)/( 3*\a^3 - 8*\a)  },{ 10*(\a^2 + \a -4)/( 3*\a^3 - 8*\a)   }) {$\bullet$};
    
    \draw[very thick] ( 0, {10/(\a*\a)} ) -- (  { 10*(-\a*\a + 3*\a -1)/(2*\a*\a - \a)}   ,   { 10*(\a-1)/(2*\a*\a - \a)}   );
    \draw[dashed] (  { 10*(-\a*\a + 3*\a -1)/(2*\a*\a - \a)}   ,   { 10*(\a-1)/(2*\a*\a - \a)}   ) -- (  { 10*(-\a*\a + 3*\a -1)/(2*\a*\a - \a)}   ,   { 3.85/\a}   );
    \node at (  { 10*(-\a*\a + 3*\a -1)/(2*\a*\a - \a)}   ,   { 3.9/\a}   ) {$\bar{x}$};

    \draw[dashed] (  {  ( 10 - 2*\a*10*(\a^2 + \a -4)/( 3*\a^3 - 8*\a) + 2*\a*(   (4-2*\a^2)/(3*\a^3-6*\a + (9*\a^6-48*\a^4+76*\a^2-32 )^0.5 ) )*10*(2*\a - 4)/( 3*\a^3 - 8*\a)  )/(  1 + 2*\a*    (4-2*\a^2)/(3*\a^3-6*\a + (9*\a^6-48*\a^4+76*\a^2-32 )^0.5 ) ) }  ,  { (10 - ( 10 - 2*\a*10*(\a^2 + \a -4)/( 3*\a^3 - 8*\a) + 2*\a*(   (4-2*\a^2)/(3*\a^3-6*\a + (9*\a^6-48*\a^4+76*\a^2-32 )^0.5 ) )*10*(2*\a - 4)/( 3*\a^3 - 8*\a)  )/(  1 + 2*\a*    (4-2*\a^2)/(3*\a^3-6*\a + (9*\a^6-48*\a^4+76*\a^2-32 )^0.5 ) ))/(2*\a)  }     ) -- (  {  ( 10 - 2*\a*10*(\a^2 + \a -4)/( 3*\a^3 - 8*\a) + 2*\a*(   (4-2*\a^2)/(3*\a^3-6*\a + (9*\a^6-48*\a^4+76*\a^2-32 )^0.5 ) )*10*(2*\a - 4)/( 3*\a^3 - 8*\a)  )/(  1 + 2*\a*    (4-2*\a^2)/(3*\a^3-6*\a + (9*\a^6-48*\a^4+76*\a^2-32 )^0.5 ) ) }    ,   { 3.95/\a}   );

    \node at (  {  ( 10 - 2*\a*10*(\a^2 + \a -4)/( 3*\a^3 - 8*\a) + 2*\a*(   (4-2*\a^2)/(3*\a^3-6*\a + (9*\a^6-48*\a^4+76*\a^2-32 )^0.5 ) )*10*(2*\a - 4)/( 3*\a^3 - 8*\a)  )/(  1 + 2*\a*    (4-2*\a^2)/(3*\a^3-6*\a + (9*\a^6-48*\a^4+76*\a^2-32 )^0.5 ) ) }    ,   { 3.9/\a}   ) {$x_s$};

    \node[red] at (0.3/\a,5.2/\a) {$\gamma_u$};
    \node[red] at (1/\a,4.58/\a) {$\gamma_s$};

    \node at (0.47/\a,4.66/\a) {$z_4$};

    \node[scale=1.2] at (2.1/\a,3.58/\a) {$\bar{\mathscr{L}}$};
    \node[scale=1.2] at (3.3/\a,4.3/\a) {$\mathscr{L}_2$};

    \node at (1.6/\a, 4.61/\a) {$\mathscr{M}_1$};
        \node at (1.6/\a, 4.33/\a) {$\mathscr{M}_2$};
    




    \draw (0, {10/(\a*\a)  }) -- (10/\a,0);

    \foreach \k in {2,...,120}{
    \draw ( {10*(\k -\a)/( \a*\k - \a  ) }  ,{ 10*(\a -1)/( \a*\a*\k - \a*\a ) } ) -- (10,0);

    }

    \begin{scope}
        \clip (10,0) -- ( {10*(2 -\a)/( \a*2 - \a  ) }  ,{ 10*(\a -1)/( \a*\a*2 - \a*\a ) } ) -- (0,0) -- (10,0);
        \foreach \n in {1,...,50}{
        \draw ( 0 , {10*(\n + 1/\a)/( (2*\n+1 )*\a )}) -- ( 10/\a , {10*\n/((2*\n+1)*\a ) } );
        }
    
    \end{scope}

  \end{scope}

\end{tikzpicture}
    \caption{Close up of the singularity set for $F_S$ near $p_2$, parameter value $\alpha=2.15$. }
    \label{fig:inclination}
\end{figure}

\begin{lemma}
    \label{lemma:L-L2}
    Let $\alpha > \alpha_2 \approx 2.127$. Let $\Gamma \subset S$ be a line segment aligned with some $v \in \mathcal{C}$. If $\Gamma$ connects $\mathscr{L}$ to $\mathscr{L}_2$, then there exists $k$ such that $H^k(\Gamma)$ contains a $v$-segment.
\end{lemma}

\begin{proof}
    Restrict $\Gamma$ to the open region bounded by $\mathscr{L},\mathscr{L}_2$. Observing Figure \ref{fig:Fsing}, points in $\Gamma$ return to $S$ over three or more iterates of $F$. Since $\alpha \geq \alpha_1$, if $\Gamma$ intersects $\mathscr{L}_3$ then the result holds with $k=4$ by Lemma \ref{lemma:Lkmapping}. Otherwise $\Gamma$ intersects the subset $\tilde{\mathscr{L}} \subset \mathscr{L}$ at a point $(1/\alpha - \alpha y, y)$ with $y_2 < y \leq 1/\alpha^2 $ (shown in bold in Figure \ref{fig:Fsing}). The endpoints of $\tilde{\mathscr{L}}$ map to $(0,y_2)$ and $(  3/\alpha -1 ,1/\alpha^2)$ under $F^3$, so that $F^3( \tilde{\mathscr{L}})$ lies entirely left of $\mathscr{I}_2$, see (\ref{eq:Ik}), if
    \[ \frac{1}{\alpha^2} \geq 2 \alpha \left(  \frac{3}{\alpha} -1  \right) -1 + \frac{1}{\alpha}, \]
    i.e. $\alpha > \left(3 + \sqrt{5} \right)/2 \approx 2.618$. In such a case $F_S(\Gamma)$ contains a segment joining $\mathscr{I}_2$ to $\mathscr{I}_2^\star$ and $H^4(\Gamma)$ contains a $v$-segment. For $\alpha \leq \left(3 + \sqrt{5} \right)/2$ the segment $F^3(\tilde{\mathscr{L}})$ intersects $\mathscr{I}_2$ at $(\bar{y},\bar{y})$,
    \begin{equation}
        \label{eq:ybar}
        \bar{y} = \frac{\alpha -1}{2\alpha^2 - \alpha},
    \end{equation}
    and $F^3(\Gamma)$ joins $\mathscr{I}_2$ to $\mathscr{I}_2^\star$ provided that $y< \bar{y}$. Analogous to before, $H^4(\Gamma)$ then contains a $v$-segment.

    It remains to consider the case where $\Gamma$ intersects the subset $\bar{\mathscr{L}} \subset \mathscr{L}$ at $(1/\alpha - \alpha y, y)$ satisfying $y \geq \bar{y}$. This segment is plotted in Figure \ref{fig:inclination}, with endpoints $(0,1/\alpha^2)$ and $(\bar{x},\bar{y})$ where
    \[ \bar{x} =  \frac{1}{\alpha} - \alpha \bar{y} = \frac{-\alpha^2 + 3 \alpha -1}{2\alpha^2 - \alpha}.  \]
    By cone alignment $\Gamma$ can then only intersect $\mathscr{L}_2$ at some point $(x, (1-x)/(2\alpha))$ with $0\leq x \leq \bar{x}$. Comparing the equation for the parent line $y-y_p = g_-(x-x_p)$ of $\gamma_s$ with those of $\partial M$ (for reference $\mathscr{M}_1$ and $\mathscr{M}_2$ lie on the lines
    \[  y= \frac{\alpha^2 -1}{-3 \alpha^3 + 4 \alpha} \left( x - \frac{1}{\alpha + 1} -1 \right) \text{ and } y= \frac{2 \alpha - \alpha^3}{1 - y \alpha^2 + 3 \alpha^4} \left( x - \frac{-\alpha^2 + \alpha + 1}{2 \alpha - \alpha^3} -1 \right)  \]
    respectively), one can verify that over the remaining parameter range $2 < \alpha \leq \left(3 + \sqrt{5} \right)/2$ the stable manifold intersects $\partial M$ on $\partial S_1$ and $\mathscr{L}_2$. Writing $(x_s,y_s) = \gamma_s \cap \mathscr{L}_2$, if $\bar{x} < x_s$ then $\Gamma$ intersects $\gamma_s$ at some point in $\Gamma^\circ$ and the result follows over $\alpha > \alpha_2$ by Lemma \ref{lemma:periodicPoint}. Indeed the inequality holds for all $\alpha > \alpha_5 \approx 2.124$, the parameter value for which $\bar{x} = x_s$.
\end{proof}

\section{Growth lemma}
\label{sec:growth}
Given a line segment $\Gamma \subset S$, we define its \emph{height} as $\ell_v(\Gamma) = \nu \left( \{y \,|\, (x,y) \in \Gamma \} \right)$ and \emph{width} as $\ell_h(\Gamma) = \nu \left( \{x \,|\, (x,y) \in \Gamma \} \right)$, where $\nu$ is the Lebesgue measure on $\mathbb{R}$.

\begin{lemma}
    \label{lemma:Fgrowth}
    Let $\alpha > \alpha_0 \approx 2.1319$. Let $\Gamma \subset S$ be a line segment aligned with some $v \in \mathcal{C}$. Either:
    \begin{enumerate}[label={(C\arabic*):}]
        \item There exists $\delta >0$ such that $F_S(\Gamma)$ contains a segment $\Gamma'$ with $\ell_h(\Gamma') > (1+ \delta) \, \ell_v(\Gamma)$, or
        \item $\Gamma$ connects $\mathscr{L}_{k-1}$ to $\mathscr{L}_k$ or $\mathscr{L}_{k-1}^\star$ to $\mathscr{L}_k^\star$ for some $k \geq 3$, or
        \item There exists $k$ such that $H^k(\Gamma)$ contains a $v$-segment.
    \end{enumerate}
\end{lemma}

\begin{proof}

Suppose first that $\Gamma$ intersects $\mathscr{L}_2$ and $\mathscr{L}_2^\star$. Then $\Gamma' = F^2(\Gamma \cap A_2) \subset F_S(\Gamma)$ is a $h$-segment and $\Gamma'' = G^2(\Gamma' \cap B_2) \subset G_S(\Gamma')$ is a $v$-segment, so that (C3) follows with $k=3$. Otherwise $\Gamma$ lies entirely below the segment $\mathscr{L}_2^\star$ or entirely above $\mathscr{L}_2$. We consider the first case now, addressing the latter case at the end.

Suppose $\Gamma$ lies entirely within some set of constant return time $k$ under $F$. The image $\Gamma' = F_S(\Gamma)$ then satisfies $\ell_h(\Gamma) \geq E_k \ell_v(\Gamma)$ where
    \begin{equation}
        \label{eq:expansionFactors}
        E_k(\alpha) := \inf_{v \in C} \frac{\|DF^k v\|_\infty}{\| v \|_\infty} = k \alpha + L
    \end{equation}
denote minimum expansion factors under the $\| \cdot \|_\infty$ norm. For all $\alpha >2$ and $k \in \mathbb{N}$ this factor is strictly greater than 1 so (C1) holds. More generally, suppose $\Gamma$ splits into multiple components $\Gamma_i$, $i \in I$, of constant return time under $F$. Denoting the list of these return times by $K = [r(\Gamma_i;,F,S) \, | \, i \in I]$, if
\begin{equation}
    \label{eq:growthCondition}
    \sum_{k \in K} \frac{1}{E_k(\alpha)} < 1
\end{equation}
then there exists $i \in I$ such that $\Gamma' = F_S(\Gamma_i)$ similarly satisfies $\ell_h(\Gamma') > \ell_v(\Gamma)$\footnote{This is a basic complexity estimate; a detailed proof is found in \cite{myers_hill_exponential_2022}.}.

Suppose $ \# K<4$. If $\Gamma$ intersects $A_1$ and $A_2$, it must connect $\mathscr{L}$ to $\mathscr{L}_2$. By Lemma \ref{lemma:L-L2}, (C3) follows. Otherwise, noting that $E_k$ is strictly increasing in $k$, the summation (\ref{eq:growthCondition}) is bounded above by that on $K = [ 1, 3, 4]$. Letting $\alpha_0 \approx 2.1319$ denote the parameter value for which
\[ \frac{1}{E_1(\alpha_0)} + \frac{1}{E_3(\alpha_0)} + \frac{1}{E_4(\alpha_0)} = 1,  \]
in the case $\# K <4$, the lemma then follows for $\alpha > \alpha_0$.

Suppose, then, that $\# K \geq 4$. If $\Gamma$ avoids the secondary accumulation sets $A_k^2$ then, noting Figure \ref{fig:Fsing}, there exists $k \geq 3$ such that $\Gamma$ intersects $A_{k-1}$, $A_k$, and $A_{k+1}$. In doing so it connects $\mathscr{L}_{k-1}$ to $\mathscr{L}_k$ so that (C2) is satisfied. Assume, then, that $\Gamma$ intersects at least one of the secondary accumulation sets. Again noting Figure \ref{fig:Fsing}, $\Gamma$ either 
\begin{enumerate}
    \item Connects $\mathscr{L}_3$ to $\mathscr{L}_3^2$, or
    \item Intersects some trio $A_{k-2}^2$, $A_k^2$, $A_{k+2}^2$ for some $k \geq 5$ (defining $A_3^2 = A_3$).
\end{enumerate}
In the first case $\Gamma' = F^3(\Gamma \cap A_3) \subset F_S(\Gamma)$ is a $h$-segment and $\Gamma'' = G^2(\Gamma' \cap B_2) \subset G_S(\Gamma')$ is a $v$-segment. Moving onto the second case, it follows that 
\begin{enumerate}[label={($\dag$):}]
        \item $\Gamma$ traverses $A_k^2$, connecting $\mathscr{L}_{k-2}^2$ to $\mathscr{L}_{k}^2$, for some $k \geq 5$.
    \end{enumerate}
By (\ref{eq:Lk2}) each $\mathscr{L}_k^2$ intersects $\partial S_1$ at the point
\begin{equation}
    \label{eq:Yk}
    (0,Y_k) = \left ( 0 , \frac{(k-1)\alpha + 2}{2k \alpha^2} \right)
\end{equation}
and $\mathscr{L}_2$ at the point
\begin{equation}
    \label{eq:YkL2}
    \left( \frac{\alpha-2}{(k-2)\alpha} , \frac{(k-3)\alpha + 2}{2 (k-2) \alpha^2} \right) = \left( \frac{\alpha-2}{(k-2)\alpha} , Y_{k-2} \right) ;
\end{equation}
see the magnified part of Figure \ref{fig:Fsing}.

Suppose that $\Gamma$ satisfying ($\dag$) violates the lemma; we will show that this leads to a contradiction by an inductive argument. To avoid satisfying (C1) the restriction $\Gamma_2 = \Gamma \cap A_2$ must satisfy
\[ \frac{\ell_v(\Gamma_2)}{\ell_v(\Gamma)} \leq \frac{1}{E_2},  \]
giving
\begin{equation}
    \label{eq:proportionEstimate}
    \frac{\ell_v(\tilde{\Gamma})}{\ell_v(\Gamma)} \geq 1 - \frac{1}{E_2}
\end{equation}
where $\tilde{\Gamma} = \Gamma \setminus \Gamma_2$. As the base stage of the induction suppose ($\dag$) holds with $k=5$, intersecting $\mathscr{L}_3^2$ and $\mathscr{L}_5^2$. Noting Lemma \ref{lemma:periodicPoint}, to violate the lemma we must have $\Gamma^\circ \cap \gamma_s = \varnothing$. This gives an upper bound
\begin{equation}
    \label{eq:baseCaseUpper}
    \ell_v(\tilde{\Gamma}) \leq \frac{1}{2\alpha} - y_m
\end{equation}
where $(x_m,y_m)$ denotes the intersection of $\gamma_s$ with the line passing through $\mathscr{L}_5^2 \cap \mathscr{L}_2$, gradient $1/L$, the lowest point along $\gamma_s$ that a segment intersecting $\mathscr{L}_5^2$ and aligned with some $v \in \mathcal{C}$ can hit. 
Using (\ref{eq:YkL2}), $y_m$ is given by
\[ y_m = \frac{ g_-( 1-2 \alpha Y_3 - x_p - L Y_3 ) + y_p  }{1-Lg_-}. \]
The shortest height of any segment aligned with $v \in \mathcal{C}$ connecting $\mathscr{L}_{k-2}^2$ to $\mathscr{L}_k^2$ is that which lies on $\partial S_1$, given by $Y_k - Y_{k-2}$. Denoting $\Gamma_5 = \Gamma \cap A_5^2$, by (\ref{eq:proportionEstimate}) then (\ref{eq:baseCaseUpper}) the image $\Gamma' = F^5(\Gamma_5) \subset F_S(\Gamma)$ satisfies
\begin{equation*}
    \begin{split}
        \ell_h(\Gamma') - \ell_v(\Gamma) & \geq \ell_h(\Gamma') - \frac{1}{1 - \frac{1}{E_2}} \ell_v(\tilde{\Gamma}) \\
        & \geq  (Y_5-Y_3) E_5 - \frac{1}{1 - \frac{1}{E_2}} \left( \frac{1}{2\alpha} - y_m \right) \\
    \end{split}
\end{equation*}
which is positive for all $\alpha > \alpha_7 \approx 2.072$. Noting $\alpha_0 > \alpha_7 $, if $\Gamma$ violates the lemma it cannot traverse $A_5^2$.

For the inductive step assume $\Gamma$ traverses $A_k^2$, but does not traverse $A_{k-2}^2$. It therefore intersects $\mathscr{L}_k^2$ and $\mathscr{L}_{k-2}^2$, but does not intersect $\mathscr{L}_{k-4}^2$. Analogous to (\ref{eq:baseCaseUpper}) this gives an upper bound
\[ \ell_v(\tilde{\Gamma}) < \frac{1}{2\alpha} - y_l  \]
where $(x_l,y_l)$ denotes the intersection of $\mathscr{L}_{k-4}^2$ with the line passing through $\mathscr{L}_k^2 \cap \mathscr{L}_2$, gradient $1/L$. Again using (\ref{eq:YkL2}), $y_l$ is given by
\begin{equation}
    \label{eq:yl}
    y_l = \frac{\left(k-4\right)\alpha Y_{k-4}-\frac{\alpha-2}{\left(k-2\right)\alpha}+LY_{k-2}}{L + \left(k-4\right)\alpha}.
\end{equation}
Analogous to the base case $\Gamma' = F^k(\Gamma \cap A_k^2)$ then satisfies
\begin{equation}
    \label{eq:inductiveEstimate}
    \ell_h(\Gamma') - \ell_v(\Gamma) > (Y_k-Y_{k-2}) E_k - \frac{1}{1 - \frac{1}{E_2}} \left( \frac{1}{2\alpha} - y_l \right) := f_\alpha(k).
\end{equation}
The function $f_\alpha(k)$ is positive\footnote{This is formally verified by noting that $f_\alpha(k)$ is continuous on $k > 9/2 \geq 4 - L/\alpha$ and shares its roots with $k(k-1)[L + (k-4)\alpha]f_\alpha(k)$, a quadratic with roots $k_1,k_2<7$ for $\alpha>\alpha_8$. This gives $\mathrm{sgn}f_\alpha(k) = \mathrm{sgn}f_\alpha(7)$ for all $k \geq 7$, with $f_\alpha(7)$ positive for all $\alpha>\alpha_8$.} for all $k \geq 7$ provided that $\alpha > \alpha_8 \approx 2.012$, the parameter value for which $f_{\alpha_8}(7) = 0$. It follows by induction that for $\Gamma$ to violate the lemma it must not traverse $A_7^2$, nor $A_9^2$, and so on. But this directly contradicts ($\dag$), so no such $\Gamma$ exists, verifying the lemma for the case where $\Gamma$ lies entirely below $\mathscr{L}_2^\star$.

The case where $\Gamma$ lies entirely above $\mathscr{L}_2$ follows similarly. The image $\Gamma^\star = I_1(\Gamma)$ lies entirely below the line $\mathscr{L}_2^\star$ and is aligned with $DI_1 v = -v \in \mathcal{C}$. By the above, $\Gamma^\star$ satisfies one of (C1-3). Noting that $I_1$ commutes with $F_S$ and preserves $\ell_h$, if $\Gamma^\star$ satisfies (C1), so does $\Gamma$. It similarly inherits (C2) or (C3) from $\Gamma^\star$, noting that $I_1$ interchanges $\mathscr{L}_k \leftrightarrow \mathscr{L}_k^\star$, commutes with $H$, and maps $v$-segments to $v$-segments.
\end{proof}

Recalling $\mathcal{H} = F \circ G$, we have an analogous result for growth under $G_S$:

\begin{lemma}
    \label{lemma:Ggrowth}
    Let $\alpha > \alpha_0 \approx 2.1319$. Let $\Lambda \subset S$ be a line segment aligned with some $v' \in \mathcal{C}'$ Either:
    \begin{enumerate}[label={(C\arabic*'):}]
        \item There exists $\delta >0$ such that $G_S(\Lambda)$ contains a segment $\Lambda'$ with $\ell_v(\Lambda') > (1+ \delta) \, \ell_h(\Lambda)$, or
        \item $\Lambda$ connects $\mathscr{I}_{k-1}$ to $\mathscr{I}_k$ or $\mathscr{I}_{k-1}^\star$ to $\mathscr{I}_k^\star$ for some $k \geq 3$, or
        \item There exists $k$ such that $\mathcal{H}^k(\Lambda$) contains a $h$-segment.
    \end{enumerate}
\end{lemma}

\begin{proof}
    Let $\Gamma = I_3^{-1}(\Lambda)$, a line segment in $S$ aligned with $v = DI_3^{-1}v' \in \mathcal{C}$. By Lemma \ref{lemma:Fgrowth}, one of (C1-3) follows. In case (C1) the segment $\Lambda' = I_3(\Gamma') \subset G_S(\Lambda)$ has height $\ell_v(\Lambda') = \ell_h(\Gamma') > (1+\delta) \, \ell_v(\Gamma) =  (1+\delta) \,\ell_h(\Lambda)$, satisfying (C1'). Cases (C2') and (C3') similarly follow from (C2) and (C3), noting that $I_3$ maps the $\mathscr{L}_k^{(\star)}$ onto the $\mathscr{I}_k^{(\star)}$, satisfies $I_3 \circ H^k = \mathcal{H}^k \circ I_3$, and maps $v$-segments into $h$-segments.
\end{proof}

\section{Proof of the main theorem}
\label{sec:proof}
\begin{proof}[Proof of Theorem \ref{thm:Bernoulli}]
    As noted in section \ref{sec:background}, it is sufficient to establish (\textbf{MR}). Given $\gamma_u(z)$, we iteratively apply Lemmas \ref{lemma:Fgrowth}, \ref{lemma:Ggrowth} to $\Gamma_0 \subset H^i\gamma_u(z) \subset S$, aligned with some $v \in \mathcal{C}$. This generates two sequences of line segments $(\Gamma_m)$, $(\Lambda_m)$ with $\Lambda_m = \Gamma_m'$ from (C1) and $\Gamma_m = \Lambda_{m-1}'$ from (C1'). Each $\Gamma_m$ lies in $H^{m'}(\Gamma_0)$ for some integer $m'$ and $\Lambda_m$ lies in $F \circ H^{m''}(\Gamma_0)$ for some integer $m'' \geq m'$. Their \emph{diameters} $\ell(\cdot) = \max \{ \ell_v(\cdot), \ell_h(\cdot) \}$ grow exponentially, so that after some finite number $m_1$ steps either $\Gamma_{m_1}$ satisfies (C2) or (C3), or $\Lambda_{m_1}$ satisfies (C2') or (C3').

    Starting with case (C3), we can find $k$ such that the image $H^k(\Gamma_{m_1}) \subset H^{m_1' + k}(\Gamma_0)$ contains a $v$-segment. Similarly in case (C3') the image $\mathcal{H}^k(\Lambda_{m_1}) \subset \mathcal{H}^k \circ F \circ H^{m_1''}(\Gamma_0) = F \circ H^{m_1'' +k}(\Gamma_0)$ contains a $h$-segment, connecting $\mathscr{I}_2$ to $\mathscr{I}_2^\star$. As seen previously, the image $H^{m_1'' +k + 2}(\Gamma_0)$ then contains a $v$-segment.

    For case (C2) write $\Gamma = \Gamma_{m_1}$; connecting $\mathscr{L}_{k-1}$ to $\mathscr{L}_k$ or $\mathscr{L}_{k-1}^\star$ to $\mathscr{L}_k^\star$ for some $k \geq 3$. If $k$ is odd, making use of the transformations $I_1$ and $I_3$, we may apply Lemma \ref{lemma:Lkmapping} some $k-3$ times until $H_S^{\frac{k-3}{2}}(\Gamma)$ contains a segment joining $\mathscr{L}_2$ to $\mathscr{L}_3$ or $\mathscr{L}_2^\star$ to $\mathscr{L}_3^\star$. The second part of Lemma \ref{lemma:Lkmapping}, together with the transformation $I_1$, ensures that we then map into a $v$-segment under $H^4$. Similarly if $k$ is even then $F_S \circ H_S^{\frac{k-4}{2}}(\Gamma)$ contains a segment $\Gamma'$ connecting $\mathscr{I}_2$ to $\mathscr{I}_3$ or $\mathscr{I}_2^\star$ to $\mathscr{I}_3^\star$ and lies in $F\circ H^{m_2}(\Gamma)$ for some integer $m_2$. Now using $I_1$, $I_3$, and Lemma \ref{lemma:Lkmapping}, the image $\mathcal{H}^4(\Gamma')$ contains a $h$-segment, connecting $\mathscr{I}_2$ to $\mathscr{I}_2^\star$. Hence $H \circ G \circ \mathcal{H}^4(\Gamma') \subset H^{6+m_2}(\Gamma)$ contains a $v$-segment. Case (C2') can then be reduced to case (C2), again making use of the transformation $I_3$.

    In any case, then, we can find $M_0$ such that $H^{M_0}(\Gamma_0)$ contains a $v$-segment $\Gamma$. Letting $\Gamma' = F^2(\Gamma \cap A_2)$, the image $\Gamma'' = G^2( \Gamma' \cap B_2) \subset H^3(\Gamma)$ is similarly a $v$-segment. Hence $H^{M_0+3}(\Gamma_0)$ contains a $v$-segment, as does $H^{M_0+3k}(\Gamma_0)$ for all $k \geq 1$ by induction. In particular $\Gamma''$ lies in $G^2(B_2)$, with endpoints on $\partial S_3$ and $\partial S_4$, bounded to the right by $I_2(\mathscr{I}_2)$ and to the left by $I_2(\mathscr{I}_2^\star)$. An example sketch is given in Figure \ref{fig:segments}. Such a segment intersects $\mathscr{L}_3$ provided $\mathscr{L}_3 \cap \mathscr{L} = ( 1/\alpha - \alpha y_2, y_2 )$ lies left of $I_2(\mathscr{I}_2^\star)$. Noting that $I_2(\mathscr{I}_2^\star)$ lies on the line $y = 1 - 2\alpha x$, this amounts to checking that
    \[ y_2 \leq 1-2\alpha \left( \frac{1}{\alpha} - \alpha y_2 \right).\]
    This is equivalent to (\ref{eq:2inequality}) so holds for all $\alpha \geq \alpha_1$. Noting $\alpha_0 > \alpha_1$, the segment $\Gamma''$ intersects $\mathscr{L}_2$ and $\mathscr{L}_3$. By Lemma \ref{lemma:Lkmapping}, $H^4(\Gamma'')$ contains a $v$-segment, similarly in $G^2(B_2)$. The integer combinations $3k+4l$ with $k\geq 1$ and $l \geq 0$ cover all integers greater than 8, so $H^m\gamma_u(z)$ contains a $v$-segment for all $m \geq M = i + M_0 + 9$.

    Given $\gamma_s(\zeta)$, we may find $i$ such that $H^{-i}\gamma_s(\zeta)$ contains a segment $\Gamma^s \subset S$, aligned with $v^s \in \mathcal{C}^s$. Now $\Gamma_0 := I_4(\Gamma^s)$ is a line segment in $S$, aligned with $v \in \mathcal{C}$. By the above we can find some integer $M_0$ such that $H^m(\Gamma_0)$ contains a $v$-segment for all $m \geq M_0 + 9$. Since $I_4^{-1}$ maps $v$-segments to $h$-segments, $H^{-n}\gamma_s(\zeta)$ = $(I_4^{-1} \circ H^n \circ I_4) \gamma_s(\zeta)$ contains a $h$-segment for all $n \geq N = i + M_0 + 9$.
\end{proof}

\section{Final remarks}
\label{sec:remarks}
As alluded to in the introduction, the lower bound $\alpha_0$ forms a natural barrier to analysis when relying on the canonical induced map $H_S$ (or rather its components $F_S, G_S$) for growth. It is the parameter value below which the `one-step expansion condition' of \cite{chernov_billiards_2005} fails for the map $F_S$ over unstable manifolds bounded away from the accumulation points $p_i^{(\star)}$ (the growth near which we ensure using the inductive argument, else\footnote{While unnecessary here, the inductive argument may similarly be applied to establish growth near $p_1^{(\star)}$.} map into $v$-segments by repeatedly applying Lemma \ref{lemma:Lkmapping}). Considering expansion under the full composition $H_S$ or its higher powers may widen the mixing window to some $\alpha_0'<\alpha_0$. This is no simple task, however, owing to the increased complexity of the singularity set. Further $\alpha_0'$ would always be bounded some distance away from the optimal shear parameter $\alpha = 2$, where $H_S$ and all its powers lose uniform hyperbolicity. At this parameter, the problematic region is $T = H^{-1}(S) \cap S $ on which $DH_S = DG \, DF = \left( \begin{smallmatrix}
    1 & 2 \\ -2 & -3
\end{smallmatrix} \right) $ is parabolic. Analogous to the map in \cite{myers_hill_loss_2023}, $T$ contains a pair of periodic line segments on which Lyapunov exponents are zero and nearby points may remain trapped for arbitrary long periods, introducing a new source of intermittent behaviour. Following \cite{myers_hill_loss_2023}, an appropriate induced map for establishing growth is the return map $H_{\sigma}$, where $\sigma = S \setminus H(T)$. While it is uniformly hyperbolic over $\alpha \geq 2$, the complexity of the singularity set likely precludes a concise analysis.

The other bound $\alpha<3$ of Theorem \ref{thm:Bernoulli} is however one of convenience, allowing for a more compact argument. While new accumulation points arise each time $\alpha$ surpasses an integer $k$, sets of constant return time $k$ under $F$ become (like $A_2$ in the present work) quadrilaterals with sides on $\partial S_1, F^{-k}(\partial S_1),\partial S_2, F^{-k}(\partial S_2)$. Any $\Gamma$ traversing these sets maps into a $v$-segment, so these sets divide up the analysis into cases analogous to those encountered here. Growth estimates are no more difficult to show, as expansion factors (\ref{eq:expansionFactors}) are universally larger with increasing $\alpha$.

\printbibliography


\end{document}